\documentclass[final]{siamart171218}

\usepackage{lipsum}
\usepackage{amsfonts}
\usepackage{amsmath}
\usepackage{amssymb}
\usepackage{graphicx}
\usepackage{epstopdf}
\usepackage{algorithmic}
\usepackage{enumerate}
\usepackage{mathtools}

\ifpdf
  \DeclareGraphicsExtensions{.eps,.pdf,.png,.jpg}
\else
  \DeclareGraphicsExtensions{.eps}
\fi

\newsiamremark{remark}{Remark}
\newsiamremark{hypothesis}{Hypothesis}
\newsiamremark{example}{Example}
\crefname{hypothesis}{Hypothesis}{Hypotheses}
\newsiamthm{claim}{Claim}

\newcommand{\opts}{\star}

\headers{Geometric Matrix Midrange}{C. Mostajeran, C. Grussler, and R. Sepulchre}

\title{Geometric Matrix Midranges \thanks{\textbf{Funding:} This work benefited from funding by the European Research Council under the Advanced ERC Grant Agreement Switchlet n.670645. C. Mostajeran is supported by the Cambridge Philosophical Society.}}

\author{Cyrus Mostajeran\thanks{Department of Engineering, University of Cambridge, United Kingdom
  (\email{csm54@cam.ac.uk})}.
\and Christian Grussler\footnotemark[2] \textsuperscript{,}\thanks{Department of Electrical Engineering and Computer Sciences, UC Berkeley, USA.}
\and Rodolphe Sepulchre\footnotemark[2]}

\usepackage{amsopn}
\DeclareMathOperator{\diag}{diag}

\ifpdf
\hypersetup{
  pdftitle={An Example Article},
  pdfauthor={D. Doe, P. T. Frank, and J. E. Smith}
}
\fi

\begin{document}

\maketitle

\begin{abstract}
We define geometric matrix midranges for positive definite Hermitian matrices and study the midrange problem from a number of perspectives. Special attention is given to the midrange of two positive definite matrices before considering the extension of the problem to more than two matrices. We compare matrix midrange statistics with the scalar and vector midrange problem and note the special significance of the matrix problem from a computational standpoint. We also study various aspects of geometric matrix midrange statistics from the viewpoint of linear algebra, differential geometry and convex optimization. A solution to the $N$-point problem is offered via convex optimization.
\end{abstract}



\section{Introduction}

The midrange of a collection of real numbers $a_1,\dots,a_N$ is defined as the arithmetic average of the extremal values. That is,
\begin{equation*}
x=\frac{1}{2}\left(\min_i a_i + \max_i a_i\right).
\end{equation*}
This is the unique solution to the optimization problem
\begin{equation*}
\min_{x\in\mathbb{R}} \; \max_i \; |x - a_i|.
\end{equation*}
In this paper, we are interested in midrange statistics in convex cones and in particular the cone of positive definite Hermitian matrices of a fixed dimension.\footnote{This manuscript further develops ideas partially introduced by the authors in \cite{GSI2019}.} The midrange of scalar-valued data is sensitive to outliers and is therefore a non-robust statistic. Despite this, it can be a useful measure in some contexts. For instance, the midrange is the maximally efficient estimator for the center of a uniform distribution. Thus, it can be an appropriate tool for data that is devoid of extreme outliers. It can also be useful in clustering algorithms that require the isolation of outlying clusters \cite{k-midranges,clustering2006,stigler2016}. It is an important notion in the statistics of extreme events \cite{Gumbel1967}.

Data representations based on symmetric positive definite matrices are common in a variety of applications from computer vision to machine learning. Often such matrices arise as covariance matrices that encode the correlations implicit in data and are thus highly structured \cite{Zhu2007}. Specific applications include brain-computer interface (BCI) systems \cite{Rao2013,Salem2018}, radar data processing \cite{Arnaudon2013}, and diffusion tensor imaging (DTI) \cite{Dryden2009}. It has been noted in many works that using nonlinear geometries related to generalized spectral properties of positive definite matrices yield significantly improved performance \cite{Pennec2006}.
Indeed, Euclidean techniques for statistics and analysis on covariance matrices often result in poor accuracy and undesirable effects, such as swelling phenomena in DTI \cite{Log-Euclidean2006}. It is in this context that much attention has been paid to developing geometric statistical methods on the cone of positive definite matrices \cite{Ando2004,Bhatia,Bhatia2006,Kubo1980,Moakher2005,Sra2012}.
A fundamental geometry that is associated to such spaces is the affine-invariant geometry \cite{Bhatia,Mostajeran2018}, whereby congruence transformations play the role of translations between matrices. The analogue of this geometry for scalars defined in the cone of positive real numbers $\mathbb{R}_+=\{x\in\mathbb{R}:x>0\}$ simply reduces to working with the logarithms of the data points and then mapping the result back to the positive cone $\mathbb{R}_+$ via the exponential map. Thus, we can define the affine-invariant midrange of $N$ positive numbers $y_i>0$ to be
\begin{equation} \label{scalar}
x=\exp\left(\frac{1}{2}\left[\min_i\log y_i+\max_i \log y_i\right]\right) = \left(\min_i y_i \cdot \max_i y_i\right)^{1/2}.
\end{equation}
Note that \cref{scalar} is the unique solution of the optimization problem
\begin{equation*}
\min_{x>0} \; \max_i \; |\log x - \log y_i| = \min_{x>0} \; \max_i \; \bigg|\log \frac{x}{y_i}\bigg|.
\end{equation*}

In the matrix setting, we define the geometric midrange problem on the cone of positive definite matrices as
\begin{align}  \label{matrix midrange}
\min_{X \succ 0} \; \max_i \; \|\log(Y_i^{-1/2}XY_i^{-1/2})\|_{\infty},
\end{align}
where $\{Y_1,\dots,Y_N\}$ are a collection of $N$ positive definite matrices of dimension $n$, $\|\cdot\|_{\infty}$ denotes the spectral operator norm on the space of Hermitian matrices of dimension $n$  defined by $\|A\|_{\infty}=\max\{|\lambda_1(A)|,\cdots,|\lambda_n(A)|\}$, and $X\succ 0$ denotes the positive definiteness of $X$.
Note that \cref{matrix midrange} can be interpreted as the smallest enclosing ball problem for the collection of data $\{Y_i\}$ in Thompson geometry. In particular, \cref{matrix midrange} can be expressed as $\min_{X \succ 0} \; \max_i \; d_{\infty}(X,Y_i)$, where $d_{\infty}$ denotes the Thompson metric (see \cref{sec:2 points}). The smallest enclosing ball problem of a finite set of points in Euclidean space was first posed by Sylvester in \cite{Sylvester1857} and is a fundamental problem in computational geometry. The problem is also known as the minimum enclosing ball, the 1-center, or the minimax optimization problem and has been studied by several authors \cite{Badoiu2008,Welzl1991}. It is an important problem that finds many applications in computer graphics and machine learning, including in collision detection, support vector clustering and similarity search \cite{Nielsen2009,Tsang2007}. On manifolds, the Riemannian smallest enclosing ball problem has been studied by Arnaudon and Nielsen in \cite{ArnaudonNielsen2013}. In particular, the authors consider 
\begin{equation} \label{R 1 center}
\min_{X \succ 0} \; \max_i \; \|\log(Y_i^{-1/2}XY_i^{-1/2})\|_{2},
\end{equation}
which is the corresponding problem with respect to the standard affine-invariant Riemannian geometry of positive definite matrices (see \cref{subsec: geodesics}). Note that  $\|\cdot\|_2$ in \cref{R 1 center} denotes the standard Frobenius norm. Although this problem is clearly closely related to \cref{matrix midrange}, there are fundamental differences between them. For instance, as proved by Afsari in \cite{Afsari2011}, there exists a unique point that minimizes the cost function in \cref{R 1 center}. In contrast, even in the case of two matrices $Y_1=A$ and $Y_2=B$, the solution to the optimization problem \cref{matrix midrange} is generally not unique. In \cref{subsec: opt}, we provide an interpretation of the Riemannian distance $d_2$ and Thompson distance $d_{\infty}$ as members of a family of affine-invariant Finsler distances.

One particular analytic solution for the geometric midrange of two positive definite matrices $A$ and $B$ that will receive special consideration is $A*B$ defined by
\begin{equation} \label{star}
A*B=\frac{1}{\sqrt{\lambda_{\min}}+\sqrt{\lambda_{\max}}}\left(B+\sqrt{\lambda_{\min}\lambda_{\max}} A\right),
\end{equation}
where $\lambda_{\max}$ and $\lambda_{\min}$ denote the maximum and minimum generalized eigenvalues of the pencil $(B,A)$, which are determined by the equation $\det(B-\lambda A)=0$. Note that $\lambda_{\max}$ and $\lambda_{\min}$ also coincide with the maximum and minimum eigenvalues of $BA^{-1}$, respectively. We will consider the properties of the expression in \cref{star} in some detail in \cref{sec:2 points}. For now, it is instructive to compare $A*B$ with the well-known geometric mean $A\#B$ of positive definite matrices $A$ and $B$ given by
\begin{equation} \label{geometric mean}
A\#B=A^{1/2}\left(A^{-1/2}BA^{-1/2}\right)^{1/2} A^{1/2}.
\end{equation}
The matrix geometric mean has been studied in great detail by several authors and is used in a variety of applications. Much research has been devoted to extending the notion of a geometric mean from two matrices to an arbitrary number of matrices and finding efficient algorithms for computing such a mean \cite{Bini2013,Iannazzo2011,Survey2012}. These include optimization-based approaches \cite{BhatiaHolbrook2006} as well as inductive sequential constructions \cite{Ando2004,Bini2010,LimInductive2014,Shuffled2017}. 

It is noteworthy that the formula \cref{star} for a matrix midrange of $A$ and $B$ is considerably less expensive to compute than the the geometric mean \cref{geometric mean}, particularly for high dimensional matrices. This is because $A*B$ mainly relies on the evaluation of extremal generalized eigenvalues that can be computed efficiently using a variety of techniques such as Krylov subspace methods \cite{Ge2016,Golub2000,Mishra2016,Stewart2002}. In contrast, the Cholesky-Schur algorithm for computing the geometric mean \cref{geometric mean} of two matrices has a complexity of $O(n^3)$ \cite{Iannazzo2011}.
Thus, we already see an important difference between the scalar and matrix midrange problems: in the scalar case, the mean and midrange of \emph{two} points are trivially the same, whereas a geometric midrange of two matrices may be much cheaper to compute than their geometric mean. The plots in \cref{fig plots} $(a)$ and $(b)$ 
illustrate how the computational cost of the midrange $A*B$ evolves with the matrix dimension as compared to the arithmetic and geometric means. The computations are based on a large number of randomly generated positive definite matrices of dimension $n=5$ to $n=10000$. \cref{fig plots} $(c)$ provides a similar comparison for the cost of computing the Thompson distance versus the Euclidean and affine-invariant Riemannian distances. The logarithms in  \cref{fig plots} $(b)$ and $(c)$ refer to the natural logarithm. 
All computations were performed on a 2017 Apple MacBook Pro laptop in MATLAB. The extremal generalized eigenvalues of $(B,A)$ that appear in $A*B$ were computed using the \texttt{eigs} function in MATLAB with the default settings, which utilizes algorithms outlined in \cite{1998arpack,Stewart2002}.

\begin{figure}
\centering
\includegraphics[width=1\linewidth]{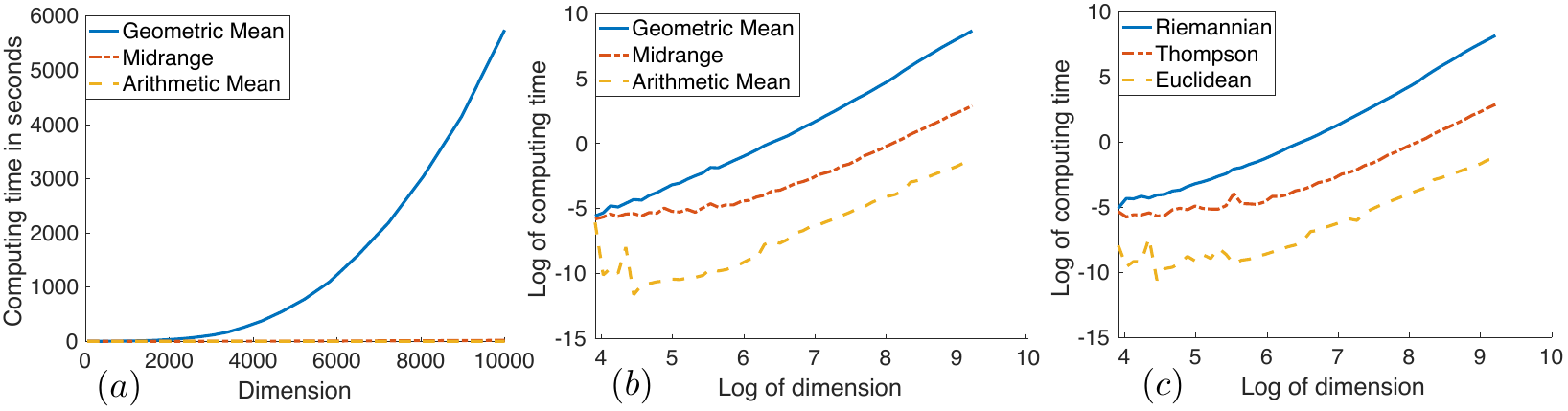}
  \caption{$(a)$, $(b):$ Computing time for the geometric mean $(A\#B)$, midrange $(A*B)$, and arithmetic mean $\left(\frac{A+B}{2}\right)$ of a pair of positive definite matrices. $(c):$ Evolution of the cost of computing Riemannian, Thompson, and Euclidean distances with matrix dimension.
   }
   \label{fig plots}
\end{figure}

\subsection{Paper organization and contributions}
The paper is organized as follows. In \cref{sec:2 points}, the midrange of two positive definite matrices is studied in detail from a variety of perspectives. We begin by proving a number of key properties of \cref{star} that are expected of a measure of central tendancy, including suitable order and monotonicity properties. In \cref{subsec: opt}, we present an interpretation of the geometric midrange within a unified optimization framework alongside the geometric mean and median. In \cref{subsec: order}, we present a characterization of the midrange formula \cref{star} based on an extremal ordering property defined using the L{\"o}wner order. In \cref{subsec: geodesics}, we review the differential geometry of the manifold of $n\times n$ positive definite Hermitian matrices and consider midranges arising as midpoints of geodesics. In \cref{sec:N point}, we define the geometric midrange problem for $N$ positive definite matrices and study its properties in some detail. We offer a solution to the problem via convex optimization in \cref{subsec: quasiconvex} before proving a number of optimality conditions and related results in \cref{subsec: optimality conditions}.

\section{Midrange of two positive definite matrices}
\label{sec:2 points}

Let $\mathbb{P}_n$ denote the set of $n\times n$ positive definite Hermitian matrices, which is the interior of the pointed, closed and convex cone of positive semidefinite matrices of the same dimensions. A pointed, closed and convex cone $C$ in a vector space $V$ induces a partial order on $V$ given by $x\leq y$ if and only if $y-x\in C$. The Thompson metric \cite{Lemmens2012,Thompson1963} on $C$ is defined to be
$d_{\infty}(x,y)=\log\max\{M(x/y;C),M(y/x;C)\}$,
where 
$M(y/x;C)=\inf\{\lambda\in\mathbb{R}:y\leq \lambda x\}$
for $x\in C\setminus\{0\}$ and $y\in V$. For $A,B\in\mathbb{P}_n$, we have $M(A/B)=\lambda_{\max}(AB^{-1})$, so that 
\begin{equation}
d_{\infty}(A,B)=\log\max\{\lambda_{\max}(AB^{-1}),\lambda_{\max}(BA^{-1})\}.
\end{equation}
Noting that $\lambda_i(A^{-1/2}BA^{-1/2})=\lambda_i(BA^{-1})$ and $\lambda_{\max}(\Sigma^{-1})=1/\lambda_{\min}(\Sigma )$ for any $\Sigma\in\mathbb{P}_n$, we find that the 2-point midrange problem \cref{matrix midrange} for data $A$ and $B$ takes the form
\begin{equation} \label{midrange Thompson}
\min_{X \succ 0} \; \max \; \{d_{\infty}(A,X),d_{\infty}(B,X)\}.
\end{equation}
A point $X$ is said to be a Thompson midpoint of the pair $(A,B)$ if $d_{\infty}(A,X)=d_{\infty}(B,X)=\frac{1}{2}d_{\infty}(A,B)$.
As $\left(\mathbb{P}_n,d_{\infty}\right)$ forms a complete metric space \cite{Lemmens2012}, the minimizers of \cref{midrange Thompson} coincide with the Thompson midpoints of $(A,B)$, which are generally non-unique. The geometry of the set of Thompson midpoints of a given pair of points $A,B\in\mathbb{P}_n$ is studied in detail in \cite{Lim2013}, where it is shown that the midpoint is unique if and only if the spectrum of $BA^{-1}$ lies in a set $\{\lambda,\lambda^{-1}\}$ for some $\lambda>0$. In this paper, we will pay special attention to the midrange $A*B$ given by \cref{star} due to its scalable computational properties.

Note that \cref{midrange Thompson} is equivalent to $\min_{X\succ 0}f(X)$, where $f(X)$ is given by
\begin{equation*}
\max\{\log\lambda_{\max}(XA^{-1}),\log\lambda_{\max}(XB^{-1}),-\log\lambda_{\min}(XA^{-1}),-\log\lambda_{\min}(XB^{-1})\}.
\end{equation*}
Using this expression and the following elementary lemma, it is easy to verify that $A*B$ is indeed a Thompson metric midpoint of $(A,B)$.
\begin{lemma} \label{shift}
If $c_1,c_2\in\mathbb{R}$ and $M$ is an $n\times n$ matrix with eigenvalues $\lambda_i(M)$, then $c_1M+c_2I$ has eigenvalues $c_1\lambda_i(M)+c_2$.
\end{lemma}
Specifically, we find that for $X=A*B$, we have
\begin{align*}
XA^{-1}&=\frac{1}{\sqrt{\lambda_{\min}}+\sqrt{\lambda_{\max}}}\left(BA^{-1}+\sqrt{\lambda_{\min}\lambda_{\max}}I\right), \\
XB^{-1} & =\frac{1}{\sqrt{\lambda_{\min}}+\sqrt{\lambda_{\max}}}\left(I+\sqrt{\lambda_{\min}\lambda_{\max}}AB^{-1}\right), 
\end{align*}
where $\lambda_{\max}$ and $\lambda_{\min}$ refer to the extremal eigenvalues of $BA^{-1}$. Using \cref{shift} and $\lambda_{\max}(BA^{-1})=1/\lambda_{\min}(AB^{-1})$, we find that $f(A*B)$ simplifies to $\frac{1}{2}d_{\infty}(A,B)$ as required.

We now consider the merits of the midrange $A*B$ as a measure of central tendency for $\{A,B\}$. The following are a number of properties that are desirable for such a mapping $\mu:\mathbb{P}_n\times\mathbb{P}_n\rightarrow\mathbb{P}_n$. We denote the conjugate transpose of $X$ by $X^*$ and the general linear group of $n\times n$ matrices by $GL(n)$.
\begin{enumerate}
\item Continuity: $\mu$ is a continuous map.
\item Symmetry: $\mu(A,B)=\mu(B,A)$ for all $A,B \in \mathbb{P}_n$.
\item Affine-invariance: $\mu(XAX^*,XBX^*)=X\mu(A,B)X^*$, for all $X\in GL(n)$.
\item Order property: $A \preceq  B \implies A\preceq  \mu(A,B) \preceq  B$.
\item Monotonicity: $\mu(A,B)$ is monotone in its arguments. 
\end{enumerate}
We will now prove that $\mu(A,B):=A*B$ indeed satisfies properties 1-3 listed above before turning our attention to the order and monotonicity properties 4 and 5, which merit special consideration. 

\begin{proposition}
The map $\mu(A,B)=A*B$ satisfies properties 1-3.
\end{proposition}

\begin{proof}
1. The continuity of $\mu$ follows directly from the expression for $A*B$ in \cref{star}, the invertibility of $A$, and the continuous dependence of eigenvalues on matrix entries, which itself follows from consideration of the roots of the characteristic polynomial of a matrix. 2. For symmetry, we note that $\lambda_{\min}(AB^{-1})=1/\lambda_{\max}(BA^{-1})$ and  $\lambda_{\max}(AB^{-1})=1/\lambda_{\min}(BA^{-1})$, so that
\begin{align*}
B*A&=\frac{1}{\sqrt{1/\lambda_{\min}}+\sqrt{1/\lambda_{\max}}}\left(A+\frac{1}{\sqrt{\lambda_{\min}\lambda_{\max}}} B\right) \\
&= \frac{\sqrt{\lambda_{\min}\lambda_{\max}}}{\sqrt{\lambda_{\min}}+\sqrt{\lambda_{\max}}}\left(\frac{1}{\sqrt{\lambda_{\min}\lambda_{\max}}} B + A\right) = A*B.
\end{align*}
3. Affine-invariance follows immediately by noting that 
\begin{equation*}
\lambda_i\left(CBC^*(CAC^*)^{-1}\right)=\lambda_i\left(CBC^*(C^*)^{-1}A^{-1}C^{-1})\right)=\lambda_{i}(BA^{-1}).  
\end{equation*}
\end{proof}

The order property is a generalization of the property of means of positive numbers whereby a mean of a pair of points is expected to lie between the two points on the number line. For Hermitian matrices, a standard partial order $\preceq $ exists according to which $A\preceq  B$ if and only if $B-A$ is positive semidefinite. This partial order is known as the L{\"o}wner order and the monotonicity in condition 5 is also with reference to this order. Unlike the case of real positive numbers $a,b>0$, which always satisfy $a\leq b$ or $b\leq a$, two Hermitian matrices $A$ and $B$ may fail to satisfy both $A\preceq  B$ and $B\preceq A$. It is well-known that the L{\"o}wner order is affine-invariant in the sense that for all $A,B\in \mathbb{P}_n$, $X\in GL(n)$,
$A\preceq  B$ implies that $XAX^* \preceq  XBX^*$.
In particular, $A\preceq  B$ if and only if $I\preceq  A^{-1/2}BA^{-1/2}$. Thus, by affine-invariance of $\mu$, it suffices to prove point 4 in the case where $A=I$ since 
$
A\preceq \mu(A,B) \preceq B$ if and only if $I \preceq \mu(I,A^{-1/2}BA^{-1/2}) \preceq A^{-1/2}BA^{-1/2}$.
To establish the 4th property for $\mu(A,B)=A*B$, we make use of \cref{shift}.
Let $\Sigma\in\mathbb{P}_n$ be such that $I\preceq  \Sigma$ and note that this is equivalent to $\lambda_i(\Sigma)\geq 1$ for $i=1,\ldots,n$. Writing $\lambda_{\min}=\lambda_{\min}(\Sigma)$, $\lambda_{\max}=\lambda_{\max}(\Sigma)$, and $\lambda_i(\Sigma)=1+\delta_i$ for $\delta_i\geq 0$, we have by \cref{shift} that
\begin{align*}
\lambda_i(I*\Sigma)-1&=\lambda_i\left(\frac{1}{\sqrt{\lambda_{\min}}+\sqrt{\lambda_{\max}}}\left(\Sigma+\sqrt{\lambda_{\min}\lambda_{\max}} I\right)\right)-1 \\
&= \frac{\lambda_i(\Sigma)+\sqrt{\lambda_{\min}\lambda_{\max}}}{\sqrt{\lambda_{\min}}+\sqrt{\lambda_{\max}}}-1 \\
&= \frac{\delta_i +\left(\sqrt{\lambda_{\min}}-1\right)\left(\sqrt{\lambda_{\max}}-1\right)}{\sqrt{\lambda_{\min}}+\sqrt{\lambda_{\max}}} \geq 0,
\end{align*}
since $\lambda_{i}(\Sigma)\geq 1$ implies that $\sqrt{\lambda_i(\Sigma)}\geq 1$. Thus, we have shown that  $I\preceq \Sigma$ implies $I \preceq  I*\Sigma$. To prove the other inequality, let $\lambda_i(\Sigma)=\lambda_{\min}(\Sigma)+\epsilon_i$ for $\epsilon_i \geq 0$, and note that
\begin{align*}
\lambda_i(\Sigma-I*\Sigma)&=\lambda_i\left(\left(\frac{\sqrt{\lambda_{\min}}+\sqrt{\lambda_{\max}}-1}{\sqrt{\lambda_{\min}}+\sqrt{\lambda_{\max}}}\right)\Sigma-\frac{\sqrt{\lambda_{\min}\lambda_{\max}}}{\sqrt{\lambda_{\min}}+\sqrt{\lambda_{\max}}} \; I \right) \\
&= \left(\frac{\sqrt{\lambda_{\min}}+\sqrt{\lambda_{\max}}-1}{\sqrt{\lambda_{\min}}+\sqrt{\lambda_{\max}}}\right)\lambda_i(\Sigma)-\frac{\sqrt{\lambda_{\min}\lambda_{\max}}}{\sqrt{\lambda_{\min}}+\sqrt{\lambda_{\max}}} \\ 
 &= \sqrt{\lambda_{\min}}\left(\sqrt{\lambda_{\min}}-1\right)+\left(\frac{\sqrt{\lambda_{\min}}+\sqrt{\lambda_{\max}}-1}{\sqrt{\lambda_{\min}}+\sqrt{\lambda_{\max}}}\right)\epsilon_i \geq 0,
\end{align*}
as $I\preceq  \Sigma$ ensures that $\sqrt{\lambda_{\min}}\geq 1$. Therefore, we have also shown that $\Sigma-I*\Sigma \succeq  0$. That is, 
\begin{equation} \label{ordering fit}
I\preceq  \Sigma \implies I \preceq  I*\Sigma \preceq  \Sigma,
\end{equation}
for all $\Sigma\in\mathbb{P}_n$. In particular, upon substituting $\Sigma=A^{-1/2}BA^{-1/2}$ in \cref{ordering fit} and using the affine-invariance properties of both the L{\"o}wner order and the mean $\mu(A,B)=A*B$, we establish the following important property.
 
\begin{proposition}
 For $A,B\in\mathbb{P}_n$, $A\preceq B$ implies that $A\preceq  A*B \preceq  B$.
\end{proposition}

We now consider the 5th and final desirable property of $\mu:\mathbb{P}_n\times\mathbb{P}_n\rightarrow\mathbb{P}_n$, which is monotonicity of $\mu$ in its arguments. First recall that a map $F:\mathbb{P}_n\rightarrow\mathbb{P}_n$ is said to be monotone if $\Sigma_1\preceq \Sigma_2$ implies that $F(\Sigma_1)\preceq F(\Sigma_2)$. By symmetry and affine-invariance, it is sufficient to consider monotonicity of $\mu(I,\Sigma)$ with respect to $\Sigma$. That is, monotonicity is established by showing that 
\begin{equation*}
\Sigma_1\preceq \Sigma_2 \implies I*\Sigma_1\preceq  \ I*\Sigma_2.
\end{equation*}
However, it turns out that $F(\Sigma):=I*\Sigma$ is not monotone with respect to $\Sigma$ as we demonstrate below. Nonetheless, $F$ is seen to enjoy certain weaker monotonicity properties. Considering the eigenvalues of $I * \Sigma$, we find that
\begin{align} \label{monotonicity eigenvalues}
\lambda_i(I*\Sigma)=\frac{\lambda_i(\Sigma)+\sqrt{\lambda_{\min}\lambda_{\max}}}{\sqrt{\lambda_{\min}}+\sqrt{\lambda_{\max}}},
\end{align}
where $\lambda_{\min}$ and $\lambda_{\max}$ refer to the smallest and largest eigenvalues of $\Sigma$. 
\begin{proposition} \label{min max monotone}
The maximum and minimum eigenvalues of $F(\Sigma)=I*\Sigma$ are monotone with respect to $\Sigma$.
\end{proposition}
\begin{proof}
Considering the cases $i=1$ and $i=n$, we find that
(\ref{monotonicity eigenvalues}) yields
\begin{equation*}
\lambda_{\min}(I*\Sigma) = \sqrt{\lambda_{\min}(\Sigma)} \quad \mathrm{and} \quad \lambda_{\max}(I*\Sigma) = \sqrt{\lambda_{\max}(\Sigma)},
\end{equation*}
both of which are seen to be monotone functions of $\Sigma$. 
\end{proof}
It is in the sense of the above that $\mu(A,B)=A*B$ inherits a weak monotonicity property. The monotonic dependence of the extremal eigenvalues of $I*\Sigma$ on $\Sigma$ ensures that if $\Sigma_1\preceq \Sigma_2$, then we can at least rule out the possibility that $I*\Sigma_1 \succ I*\Sigma_2$, where $\succ 0$ here denotes positive definiteness. To prove that monotonicity is generally not satisfied in the full sense, consider a diagonal matrix $\Sigma=\operatorname{diag}(a,b,x)\in\mathbb{P}(3)$, where $\lambda_{\min}(\Sigma)=a < b \leq x = \lambda_{\max}(\Sigma)$ and $x$ is thought of as a variable. We have $I*\Sigma = \operatorname{diag}\left(\sqrt{a},f(x),\sqrt{x}\right)$,
where 
\begin{equation*}
\lambda_2(I*\Sigma)=f(x):=\frac{b+\sqrt{ax}}{\sqrt{a}+\sqrt{x}}.
\end{equation*}
Taking the derivative of $f$ with respect to $x$, we find that 
\begin{equation*}
f'(x)=\frac{a-b}{2\sqrt{x}(\sqrt{a}+\sqrt{x})^2} < 0, \quad \forall x \geq b,
\end{equation*}
which shows that the second eigenvalue of $I*\Sigma$ decreases as $x$ increases. Thus, we see that $I*\Sigma$ cannot depend monotonically on $\Sigma$ in this example. 

As a summary, we collect the main results so far in the following theorem.

\begin{theorem}
The midrange $\mu(A,B)=A*B$ defined in \cref{star} yields a Thompson metric midpoint of $A,B\in\mathbb{P}_n$ that is continuous, symmetric and affine-invariant. Moreover, if $A\preceq  B$, then $A\preceq \mu(A,B) \preceq  B$, and the extremal eigenvalues of $\mu(I,\Sigma)$ depend monotonically on $\Sigma\in\mathbb{P}_n$.
\end{theorem}

We also note that $A*B$ satisfies a key scaling property which suggests that it may be a plausible candidate for a computationally scalable substitute for the standard geometric mean $A\#B$ of two positive definite matrices.

\begin{proposition}
For any real scalars $a,b>0$ and matrices $A,B\in\mathbb{P}_n$, we have 
\begin{equation} \label{geo scaling prop}
(aA)*(bB)=\sqrt{ab}(A*B).
\end{equation}
\end{proposition}
\begin{proof}
The result follows upon substituting $\lambda_i\left((bB)(aA)^{-1}\right)=\frac{b}{a}\lambda_i(BA^{-1})$ into the formula \cref{star}.
\end{proof}

\begin{remark}
The scaling in \cref{geo scaling prop} of course does not generally hold for a mean of two matrices. Indeed, it does not generally hold for means arising as $d_{\infty}$-midpoints either. For instance, \cite{Lim2013} identifies 
\begin{equation} \label{diamond}
A\diamond B = 
\begin{dcases}
\frac{\sqrt{\lambda_{\max}}}{1+\lambda_{\max}}(A+B) \quad \mathrm{if} \quad \lambda_{\min}\lambda_{\max} \geq 1 \\
\frac{\sqrt{\lambda_{\min}}}{1+\lambda_{\min}}(A+B) \quad \mathrm{if} \quad \lambda_{\min}\lambda_{\max} \leq 1
\end{dcases}
\end{equation}
as another $d_{\infty}$-midpoint of $A$ and $B$. Clearly $A\diamond B$ does not scale geometrically in the sense of \cref{geo scaling prop}.
\end{remark}

\subsection{An optimization-based formulation}
\label{subsec: opt}

A norm $\|\cdot\|$ on the space of $n\times n$ complex matrices is said to be unitarily invariant if $\|UXV\|=\|X\|$ for all $n\times n$ matrices $X$ and unitary matrices $U,V$. A norm $\Phi$ on $\mathbb{R}^n$ is called a symmetric gauge norm if it is invariant under permutations and sign changes of coordinates. Consider the family of affine-invariant metric distances $d_{\Phi}$ on $\mathbb{P}_n$ defined as 
\begin{equation} \label{d Phi}
d_{\Phi}(A,B)=\|\log A^{-1/2}BA^{-1/2} \|_{\Phi},
\end{equation}
where $\|\cdot\|_{\Phi}$ is any unitarily invariant norm on the space of Hermitian matrices of dimension $n$ defined by
$
\|X\|_{\Phi}:=\Phi(\lambda_1(X),\ldots,\lambda_n(X)),
$
with $\lambda_{\min}(X)=\lambda_n(X)\leq \ldots \leq \lambda_1(X) = \lambda_{\max}(X)$ denoting the $n$ real eigenvalues of $X$ and
$\Phi$ a symmetric gauge norm on $\mathbb{R}^n$ \cite{Bhatia2003}. The norms $\|\cdot\|_{\Phi}$ induced by the $l_p$-norms on $\mathbb{R}^n$ for $1\leq p \leq \infty$ are called the Schatten $p$-norms. For the choice of $\Phi(x_1,\ldots,x_n)=(\sum_i x_i^2)^{1/2}$, $d_2:=d_{\Phi}$ corresponds to the metric distance generated by the standard affine-invariant Riemannian metric on $\mathbb{P}_n$ given by $\langle X,Y \rangle_{\Sigma}=\operatorname{tr}(\Sigma^{-1}X\Sigma^{-1}Y)$ for $\Sigma\in\mathbb{P}_n$ and Hermitian matrices $X,Y\in T_{\Sigma}\mathbb{P}_n$. The length element $ds$ of this geometry satisfies
$ds^2 = \operatorname{tr}\left(\Sigma^{-1}d\Sigma\right)^2$.
The unique (up to parametrization) Riemannian geodesic from $A$ to $B$ is given by the curve $\gamma_{\mathcal{G}}:[0,1]\rightarrow\mathbb{P}_n$ defined by 
\begin{equation} \label{R geodesic}
\gamma_{\mathcal{G}}(t)=A^{1/2}\left(A^{-1/2}BA^{-1/2}\right)^t A^{1/2}.
\end{equation}
This curve is significant as a minimal geodesic for \emph{any} of the affine-invariant metrics $d_{\Phi}$ \cite{Bhatia2003}. The midpoint of $\gamma_{\mathcal{G}}$ is the matrix geometric mean $A\#B$ \cref{geometric mean}, which is a metric midpoint in the sense that $d_{\Phi}(A,A\#B)=d_{\Phi}(A\#B,B)=\frac{1}{2}d_{\Phi}(A,B)$ for any choice of symmetric gauge $\Phi$. 
With $\Phi(x_1,\ldots,x_n)=\max_i|x_i|$, $d_{\Phi}=d_{\infty}$ yields the distance function that coincides with the Thompson metric \cite{Thompson1963} on the cone $\mathbb{P}_n$
\begin{equation*} 
d_{\infty}(A,B)=\|\log A^{-1/2}BA^{-1/2}\|_{\infty}=\max\{\log\lambda_{\max}(BA^{-1}),\,\log\lambda_{\max}(AB^{-1})\}.
\end{equation*}
Therefore, we see that the geometric mean $A\#B$ is also a geometric midrange of $A$ and $B$.

The invariant Finsler metrics \cref{d Phi} provide a route to geometrically generalize several measures of aggregation of data to the space of positive definite matrices $\mathbb{P}_n$. Specifically, the mean, median, and midrange of a collection of real numbers $a_1,\dots,a_N$ can be defined as 
\begin{align*}
&\operatorname{argmin}_{x\in\mathbb{R}}\left(\sum_i (x-a_i)^2\right)^{1/2}=\operatorname{argmin}_{x\in\mathbb{R}}\|x\boldsymbol{1}-\boldsymbol{a}\|_2, \\
&\operatorname{argmin}_{x\in\mathbb{R}}\sum_i |x-a_i| = \operatorname{argmin}_{x\in\mathbb{R}}\|x\boldsymbol{1}-\boldsymbol{a}\|_1, \\
&\operatorname{argmin}_{x\in\mathbb{R}} \; \max_i \; |x - a_i| =\operatorname{argmin}_{x\in\mathbb{R}}\|x\boldsymbol{1}-\boldsymbol{a}\|_{\infty},
\end{align*}
respectively, where $\boldsymbol{1}=(1,\dots,1)\in\mathbb{R}^N$ and $\boldsymbol{a}=(a_1,\dots,a_N)$. By analogy, one can extend these notions to geometric averages for a collection of data $Y_1,\dots,Y_N\in\mathbb{P}_n$ arising as 
\begin{align} \label{general opt}
\operatorname{argmin}_{X \succ 0}\Phi_N\left(d_{\Phi_n}(X,Y_i))\right),
\end{align}
where $\Phi_n$ denotes the gauge norm on the space of $n\times n$ Hermitian matrices and $\Phi_N$ denotes the corresponding gauge function acting on the $N$ distances $d_{\Phi_n}(X,Y_i)$. If $\Phi$ corresponds to the $l_2$ vector norm, \cref{general opt} yields the \emph{geometric mean} $\mathcal{G}_2$ of $Y_1,\dots,Y_N$, also known as the Karcher mean \cite{Bhatia,BhatiaHolbrook2006,Moakher2005}:
\begin{equation} \label{Karcher}
\mathcal{G}_2(Y_1,\dots,Y_N)=\operatorname{argmin}_{X \succ 0} \sum_{i=1}^Nd_{2}(X,Y_i)^2.
\end{equation}
If $N=2$, the unique solution $\mathcal{G}_2(A,B)$ of  \cref{Karcher} coincides with the geometric mean $A\#B$. One can also define \emph{geometric medians} of $Y_1,\dots,Y_N$ to be solutions to \cref{general opt} for the choice of $\Phi(x_1,\dots,x_n)=\sum_i|x_i|$:
\begin{equation}
\mathcal{G}_1(Y_1,\dots,Y_N)=\operatorname{argmin}_{X \succ 0} \sum_{i=1}^Nd_{1}(X,Y_i).
\end{equation}
Note that the $d_1$ distance of $X\in \mathbb{P}_n$ to the identity $I$ takes the form
\begin{equation} \label{d1 logdet}
d_1(X,I)=\|\log X\|_1 = \sum_{i=1}^N|\log \lambda_i (X)| = \operatorname{tr}\left((\log X\log X)^{1/2}\right).
\end{equation}
It is interesting to compare \cref{d1 logdet} to the function $F(X)=\log \det(X)$, which plays an important role in convex optimization \cite{Boyd2004}. In particular, we have
\begin{equation*}
F(X)=\log \det(X)=\operatorname{tr}(\log X)=\sum_{i=1}^N\log \lambda_i (X) .
\end{equation*}
If $X \succeq I$, then $\log X \succeq 0$ and hence $d_1(X,I)=\operatorname{tr}(\log X)= \log \det X$.

If $\Phi$ corresponds to the $l_{\infty}$-norm, then \cref{general opt} yields the geometric midrange problem \cref{matrix midrange}. In the $N=2$ case, we have already seen that $A*B$ is a solution to the corresponding midrange optimization problem $\operatorname{min}_{X \succ 0} \max\{d_{\infty}(X,A),d_{\infty}(X,B)\}$.
The $N$-point problem is studied in more detail in \cref{sec:N point}.

\subsection{Extremal ordering property}
\label{subsec: order}

Here we describe a characterization of the midrange $A*B$ of $A,B\in\mathbb{P}_n$ that does not rely on any additional structures on $\mathbb{P}_n$ except for the standard L{\"o}wner partial order $\succeq$. It is remarkable that such a characterization that is independent of any metric or differential geometric structure on $\mathbb{P}_n$ exists.

\begin{theorem}
Let $A,B\in\mathbb{P}_n$. Then,
\begin{align*}
A*B &= \max_{X\in\operatorname{span}\{A,B\}}\bigg\{X=X^*: 
\begin{pmatrix}
A & X \\
X & B
\end{pmatrix} \succeq 0
\bigg\} \\
&=  \max_{a,b\in\mathbb{R}}\bigg\{aA+bB: 
\begin{pmatrix}
A & aA+bB \\
aA+bB & B
\end{pmatrix} \succeq 0
\bigg\}.
\end{align*}
\end{theorem}

\begin{proof}
For any $X=X^*$, we have the congruence relation
\begin{align*}
\begin{pmatrix}
A & X \\
X & B
\end{pmatrix}  &\sim 
\begin{pmatrix}
I & -XB^{-1} \\
0 & I
\end{pmatrix}  
\begin{pmatrix}
A & X \\
X & B
\end{pmatrix}
\begin{pmatrix}
I & 0 \\
-B^{-1}X & I
\end{pmatrix} \\
 &=
 \begin{pmatrix}
A-XB^{-1}X & 0 \\
0 & B
\end{pmatrix}.  
\end{align*}
This matrix is clearly positive semidefinite if and only if $A \succeq XB^{-1}X$, which is equivalent to
\begin{equation} \label{order1}
B^{-1/2}AB^{-1/2} \succeq B^{-1/2}XB^{-1}XB^{-1/2} = (B^{-1/2}XB^{-1/2})^2.
\end{equation}
Using the monotonicity of the square root function $\Sigma \mapsto \Sigma^{1/2}$ on $\mathbb{P}_n$ and the affine-invariance of the L{\"o}wner order, \cref{order1} holds if and only if
\begin{equation} \label{order2}
X \preceq B^{1/2}(B^{-1/2}AB^{-1/2})^{1/2}B^{1/2}.
\end{equation}
The expression on the right hand side is of course the geometric mean $A\#B$ and thus we have
\begin{equation*}
A\#B = \max_{X\succeq 0}\bigg\{X=X^*: 
\begin{pmatrix}
A & X \\
X & B
\end{pmatrix} \succeq 0
\bigg\}.
\end{equation*}
If we restrict $X$ to be in the real span of $A$ and $B$, \cref{order2} becomes
\begin{equation*}
aA + bB \preceq B^{1/2}(B^{-1/2}AB^{-1/2})^{1/2}B^{1/2},
\end{equation*}
which is equivalent to $a\Sigma + b I \preceq \Sigma^{1/2}$
for $\Sigma = B^{-1/2}AB^{-1/2}$. By diagonalizing $\Sigma^{1/2}$ we obtain a unitary matrix $V$ such that $\Sigma=V^*DV$ and $\Sigma^{1/2}=V^*D^{1/2}V$, where $D=\diag(\lambda_1(\Sigma),\dots,\lambda_n(\Sigma))$. Therefore, we have 
$aD + bI \preceq D^{1/2}$,
which is equivalent to
\begin{equation} \label{order3}
a\lambda_i(\Sigma) + b \leq \sqrt{\lambda_i(\Sigma)},
\end{equation}
for $i=1,\dots,n$. If we require that equality hold in \cref{order3} for $i=1$ and $i=n$, so that
\begin{equation*}
a\lambda_{\min}(\Sigma)+b = \sqrt{\lambda_{\min}(\Sigma)} \quad \mathrm{and} \quad 
a\lambda_{\max}(\Sigma)+b = \sqrt{\lambda_{\max}(\Sigma)} 
\end{equation*}
we find that 
\begin{align*}
a &= \frac{\sqrt{\lambda_{\max}(\Sigma)}-\sqrt{\lambda_{\min}(\Sigma)}}{\lambda_{\max}(\Sigma)-\lambda_{\min}(\Sigma)} = \frac{1}{\sqrt{\lambda_{\max}(\Sigma)}+\sqrt{\lambda_{\min}(\Sigma)}} \\
b &= \frac{\sqrt{\lambda_{\min}(\Sigma)\lambda_{\max}(\Sigma)}}{\sqrt{\lambda_{\max}(\Sigma)}+\sqrt{\lambda_{\min}(\Sigma)}}.
\end{align*}
For this choice of $a$ and $b$, we have $aA + bB = A*B$. Moreover, \cref{order3} is satisfied for each $i=1,\dots,n$ since
\begin{align*}
\sqrt{\lambda_i(\Sigma)}-a\lambda_i(\Sigma)-b &= \frac{\sqrt{\lambda_{\min}\lambda_i}+\sqrt{\lambda_{\max}\lambda_i}-\lambda_i-\sqrt{\lambda_{\min}\lambda_{\max}}}{\sqrt{\lambda_{\min}}+\sqrt{\lambda_{\max}}} \\
& = \frac{(\sqrt{\lambda_{\max}}-\sqrt{\lambda_i})(\sqrt{\lambda_i}-\sqrt{\lambda_{\min}})}{\sqrt{\lambda_{\max}}+\sqrt{\lambda_{\min}}} \geq 0.
\end{align*} 
Imposing equality in \cref{order3} for any pair of indices other than $i=1$ and $i=n$ would yield coefficients $a$ and $b$ that result in the violation of some of the other inequalities in \cref{order3}. Therefore, our choice of $a$ and $b$ is indeed optimal.
\end{proof}

\subsection{Differential geometric viewpoint}
\label{subsec: geodesics}

The set $\mathbb{P}_n$ is a smooth manifold whose tangent space $T_{\Sigma}\mathbb{P}_n$ at any point $\Sigma\in\mathbb{P}_n$ can be identified with the set of $n\times n$ Hermitian matrices $\mathbb{H}_n$. The matrix exponential map $X\mapsto e^X$ maps $\mathbb{H}_n$ bijectively onto $\mathbb{P}_n$. Its differential $de^X:\mathbb{H}_n\rightarrow \mathbb{H}_n$ at $X$ is the linear map given by 
$de^{X}(Z)=\frac{d}{dt}(e^{X+tZ})\vert_{t=0}$.
The following exponential metric increasing property is established by Bhatia in \cite{Bhatia2003}. See \cite{Lang1999} for an earlier version of the theorem.
\begin{theorem} \label{exp increasing}
For any symmetric gauge norm $\Phi$ and Hermitian matrices $X$ and $Z$, we have
\begin{equation*}
\|Z\|_{\Phi} \leq \|e^{-X/2}de^{X}(Z)e^{-X/2}\|_{\Phi},
\end{equation*}
where $\|\cdot\|_{\Phi}$ denotes the unitarily invariant norm induced by $\Phi$.
\end{theorem}

This theorem has several important consequences, which we will briefly review. First note that the distance functions $d_{\Phi}$ defined in \cref{d Phi} are induced by the affine-invariant Finsler structures on $\mathbb{P}_n$ given by
\begin{equation} \label{Finsler structure}
\|d\Sigma\|_{\Sigma,\Phi} := \|\Sigma^{-1/2}d\Sigma\Sigma^{-1/2}\|_{\Phi},
\end{equation}
for $\Sigma\in\mathbb{P}_n$ and $d\Sigma\in T_{\Sigma}\mathbb{P}_n$. For our purposes, we can think of a Finsler structure on $\mathbb{P}_n$ as a smoothly varying norm on the tangent bundle of $\mathbb{P}_n$. Such a structure can be used to calculate the length of any smooth curve $\gamma$ in $\mathbb{P}_n$. We can express any such curve as the image of a curve $\Gamma$ in $\mathbb{H}_n$ under the exponential map. In particular, any smooth curve $\gamma:[0,1]\rightarrow\mathbb{P}_n$ from $I$ to $\Sigma$ can be expressed as $\gamma(t)=e^{\Gamma(t)}$, where $\Gamma(0)=0$ and $\Gamma(1)=\log(\Sigma)\in\mathbb{H}_n$. The length of this curve with respect to the Finsler structure \cref{Finsler structure} is 
\begin{align*}
L_{\Phi}[\gamma] = \int_0^1 \|\gamma'(t)\|_{\gamma(t),\Phi}dt &=\int_0^1\|\gamma(t)^{-1/2}\gamma'(t)\gamma(t)^{-1/2}\|_{\Phi}dt \\
&= \int_0^1 \|\gamma(t)^{-1/2}de^{\Gamma(t)}(\Gamma'(t))\gamma(t)^{-1/2}\|_{\Phi}dt \\
&\geq \int_0^1\|\Gamma'(t)\|_{\Phi}dt.
\end{align*}
The last integral is simply the length of the curve $\Gamma$ in $\mathbb{H}_n$ and the least value it can take is $\|\log\Sigma\|_{\Phi}$, which is attained by the straight line segment from $0$ to $\log\Sigma$ in $\mathbb{H}_n$. The distance between $I$ and $\Sigma$ is defined as
$d_{\Phi}(I,\Sigma) = \inf_{\gamma} L_{\Phi}[\gamma]$,
where the infimum is taken over all smooth curves $\gamma$ from $I$ to $\Sigma$. Therefore, we see that
\begin{equation*}
d_{\Phi}(I,\Sigma) = \|\log \Sigma\|_{\Phi},
\end{equation*}
and note that this distance is attained by the curve $\gamma(t)=e^{t\log\Sigma}$, which is a geodesic from $I$ to $\Sigma$. See \cref{fig exp}.

\begin{figure}
\centering
\includegraphics[width=0.75\linewidth]{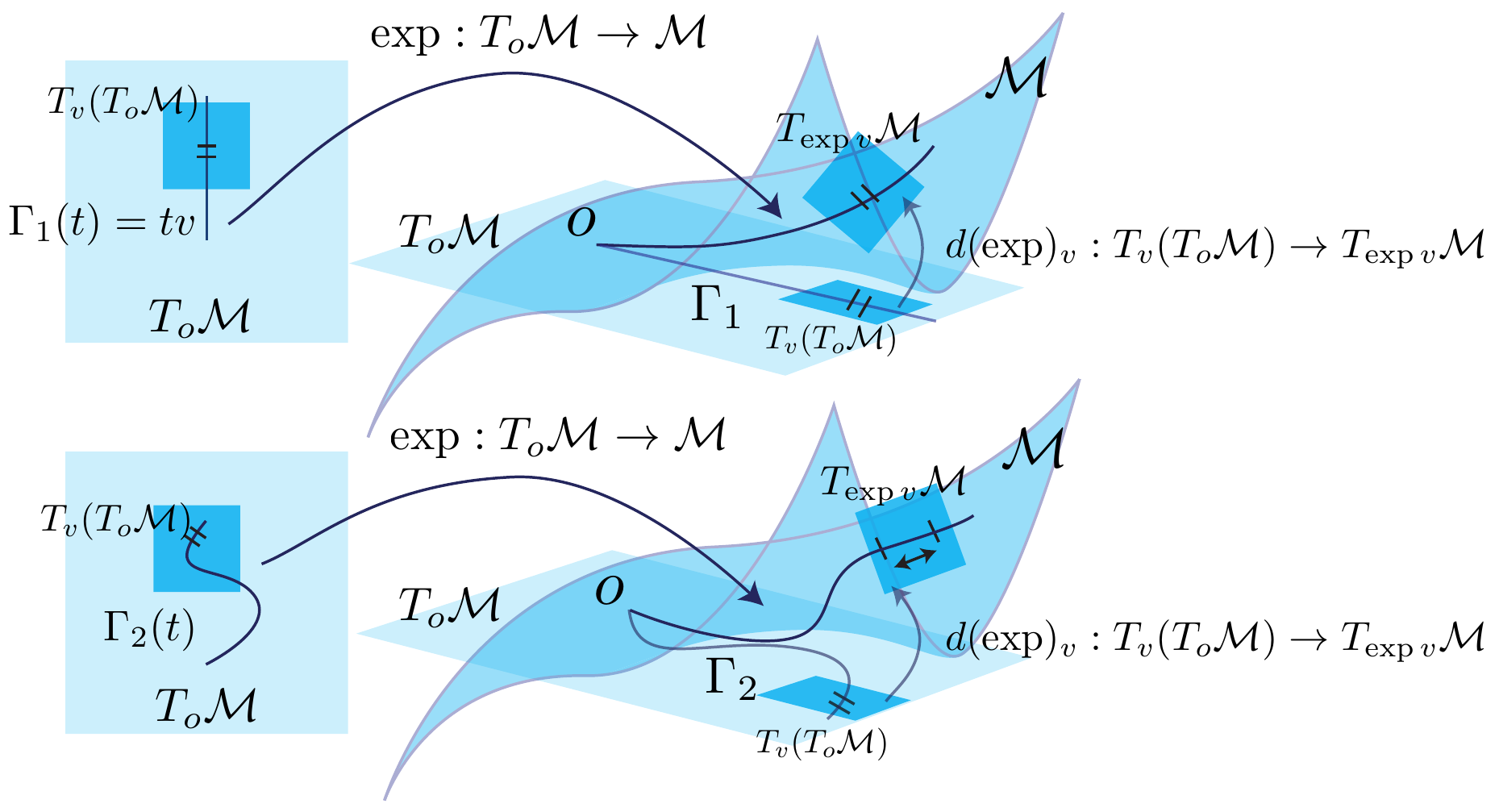}
  \caption{The figure depicts a manifold $\mathcal{M}$ whose exponential map preserves the length of rays through the origin (top), but generally increases the length of curves (bottom) as measured with respect to some Finsler structure on $\mathcal{M}$. $(\mathbb{P}_n,d_{\Phi})$ is a manifold that satisfies such an exponential metric increasing property.
   }
   \label{fig exp}
\end{figure}

Since  congruence transformations are isometries of $(\mathbb{P}_n,d_{\Phi})$, it follows that the curve $\gamma_{\mathcal{G}}^{A,B}(t) = A^{1/2}\exp(t\log(A^{-1/2}BA^{-1/2}))A^{1/2}$,
is a geodesic from $A$ to $B$. Note that this is precisely in agreement with \cref{R geodesic}. 
This geodesic is unique provided that the geodesics in $\mathbb{R}^n$ induced by $\Phi$ are unique. In particular, uniqueness of geodesics in $(\mathbb{P}_n,d_{\Phi})$ is inherited from $\mathbb{R}^n$ when $\Phi$ corresponds to the $l_p$-norms for $1<p<\infty$, but not for $p=1,\infty$. 

The exponential metric increasing property can also be used to show that the metric space $(\mathbb{P}_n,d_{\Phi})$ is a space of non-positive curvature for any choice of $\Phi$. See \cite{Bhatia2003,Bridson2013} for further details. A closely related result is the geodesic convexity \cite{Lawson2007} of $d_{\Phi}$, which follows from the inequality $d_{\Phi}(e^{tX},e^{tZ})\leq t d_{\Phi}(e^{X},e^{Z})$, 
for all $X,Z\in\mathbb{H}_n$ and $0\leq t \leq 1$. More generally, the geodesic convexity theorem states that for all $A_1,A_2,B_1,B_2\in\mathbb{P}_n$, the real function
\begin{equation} \label{geodesic convexity 1}
t \mapsto d_{\Phi}\left(\gamma_{\mathcal{G}}^{A_1,A_2}(t),\gamma_{\mathcal{G}}^{B_1,B_2}(t)\right)
\end{equation}
is convex for any symmetric gauge norm $\Phi$ \cite{Bhatia2003}.

In geometric midrange statistics we are interested in the distance $d_{\infty}$, which coincides with the Thompson metric on the cone of positive definite Hermitian matrices.
It is known that the Thompson metric does not admit unique minimal geodesics. Indeed, a remarkable construction by Nussbaum in \cite{Nussbaum1994} describes a family of geodesics that generally consists of an infinite number of curves connecting a pair of points in a cone $C$. In particular, setting $\alpha:=1/M(x/y;C)$ and $\beta:=M(y/x;C)$, the curve $\phi:[0,1]\rightarrow C$ given by
\begin{equation} \label{Nussbaum}
\phi(t;x,y):=\begin{dcases}
\left(\frac{\beta^t-\alpha^t}{\beta-\alpha}\right)y+\left(\frac{\beta\alpha^t-\alpha\beta^t}{\beta-\alpha}\right)x \quad &\mathrm{if} \; \alpha\neq\beta, \\
\alpha^t x &\mathrm{if} \; \alpha=\beta,
\end{dcases}
\end{equation}
is always a minimal geodesic from $x$ to $y$ with respect to the Thompson metric. The curve $\phi$ defines a projective straight line in the cone. If we take $C$ to be the cone of positive semidefinite matrices with interior $\operatorname{int} C= \mathbb{P}_n$, then for a pair of points $A,B\in\mathbb{P}_n$, we have $\beta=M(B/A;C)=\lambda_{\max}(BA^{-1})$ and $\alpha=1/M(A/B;C)=\lambda_{\min}(BA^{-1})$. Thus, the minimal geodesic described by \cref{Nussbaum} takes the form
\begin{equation} \label{Nussbaum matrix}
\phi(t):=\begin{dcases}
\left(\frac{\lambda_{\max}^t-\lambda_{\min}^t}{\lambda_{\max}-\lambda_{\min}}\right)B+\left(\frac{\lambda_{\max}\lambda_{\min}^t-\lambda_{\min}\lambda_{\max}^t}{\lambda_{\max}-\lambda_{\min}}\right)A &\mathrm{if} \; \lambda_{\min}\neq\lambda_{\max}, \\
\lambda_{\min}^t A &\mathrm{if} \; \lambda_{\min}=\lambda_{\max},
\end{dcases}
\end{equation}
where $\lambda_{\max}$ and $\lambda_{\min}$ denote the largest and smallest eigenvalues of $BA^{-1}$, respectively. Taking the midpoint $t=1/2$ of this geodesic, we recover the $d_{\infty}$-midpoint $A*B$ \cref{star} of $A$ and $B$. Thus, we have arrived at another interpretation of $A*B$ as the midpoint of a suitable geodesic in $(\mathbb{P}_n,d_{\infty})$. 
The result follows from elementary algebraic simplification upon setting $A*B=\phi(1/2;A,B)$ in the case $\lambda_{\min}\neq \lambda_{\max}$. If $\lambda_{\min}=\lambda_{\max}$, then $\phi(1/2;A,B)=\sqrt{\lambda_{\min}}A$ also agrees with the formula in \cref{star}. It is shown in \cite{Lim2013} that $\phi=\phi(t;A,B)$ is the unique $d_{\infty}$ geodesic connecting $A$ to $B$ if and only if the spectrum of $BA^{-1}$ consists of at most two distinct eigenvalues, one of which is the reciprocal of the other. Moreover, it is shown that otherwise there are infinitely many $d_{\infty}$ minimal geodesics from $A$ to $B$, and that the set of $d_{\infty}$-midpoints of $A$ and $B$ is compact and convex in both Riemannian and Euclidean senses \cite{Lim2013}.

\section{The $N$-point geometric midrange problem} 
\label{sec:N point}

Given a collection of $N$ points $Y_1,\dots,Y_N$ in $\mathbb{P}_n$, the midrange problem can be formulated as the following optimization problem
\begin{align}  \label{N geo midrange 1}
\min_{X \succ0} \; \max_i \; d_{\infty}(X,Y_i).
\end{align}
We call a solution $X^{\opts}$ to the above problem a midrange of $\{Y_i\}$. Note that the cost function $f(X):=\max_i d_{\infty}(X,Y_i)$ is continuous but not smooth. That is, \cref{N geo midrange 1} is a non-smooth continuous optimization problem on a smooth Finsler manifold.

\begin{proposition}
The optimum cost $t^\opts = \min_{X\succ 0}\max_id_{\infty}(X,Y_i)$ of \cref{N geo midrange 1} satisfies $l \leq t^\opts \leq u$, where the lower and upper bounds are given by
\begin{equation}
l = \frac{1}{2}\operatorname{diam}_{\infty}(\{Y_i\}):=\frac{1}{2}\max_{i,j}d_{\infty}(Y_i,Y_j), \quad u = \min_i \max _j d_{\infty}(Y_i,Y_j) \leq 2l.
\end{equation}
\end{proposition}

\begin{proof}
Let $X^\opts$ denote a midrange of $\{Y_i\}$ so that $t^\opts = \max_i d_{\infty}(X^\opts,Y_i)$.
By the triangle inequality, we have for any $i,j = 1, \ldots, N$,
\begin{equation*}
d_{\infty}(Y_i,Y_j) \leq d_{\infty}(Y_i,X^\opts) + d_{\infty}(X^\opts,Y_j) \leq t^\opts + t^\opts = 2t^\opts.
\end{equation*}
Taking the maximum of the left-hand side over $i,j$, we arrive at $l= \frac{1}{2}\operatorname{diam}_{\infty}(\{Y_i\})\leq t^\opts$. For the upper bound, note that taking $X=Y_i$ for each $i$, we obtain a cost $f(Y_i)=\max_j d_{\infty}(Y_i,Y_j)$. The minimum value of these $N$ cost evaluations will clearly still yield an upper bound on the optimum cost $t^\opts$. Thus we have $t^\opts\leq u =  \min_i \max _j d_{\infty}(Y_i,Y_j)$.
\end{proof}

Note that it is possible to have a collection of points $\{Y_i\}$ for which either $l$ or $u$ is attained. For instance, $l$ is clearly attained by the $d_{\infty}$-midpoint when $\{Y_i\}$ consists of a pair of points. Similarly, $u$ is attained if we have 3 points $Y_1,Y_2,Y_3$, where $Y_3$ happens to be a $d_{\infty}$-midpoint of $Y_1$ and $Y_2$. In general, the upper bound is attained if the midrange coincides with one of the data points.

It is instructive to consider the $N$-point affine-invariant midrange of vectors in the positive orthant. In the vector case, the midrange problem in $\mathbb{R}^n_+$ takes the form
\begin{equation} \label{vector midrange problem}
\min_{\boldsymbol{x}>0} \; \max_i \; \|\log\boldsymbol{x}-\log\boldsymbol{y}_i\|_{\infty} := \min_{\boldsymbol{x}>0} \; \max_i \; \max_a \;|\log x^a-\log y^a_i| ,
\end{equation}
where $\boldsymbol{x}>0$ means that $\boldsymbol{x}=(x^a)$ satisfies $x^a>0$ for $a=1,\ldots,n$ and $\boldsymbol{y_i}$ are a collection of $N$ given points in $\mathbb{R}^n_+$. As in the matrix case, the optimum cost $t^\opts=\min_{\boldsymbol{x}>0}f(\boldsymbol{x})=\min_{\boldsymbol{x>0}}\max_i \|\log\boldsymbol{x}-\log\boldsymbol{y}_i\|_{\infty}$ has a lower bound 
\begin{equation} \label{lower vector}
l=\frac{1}{2}\max_{i,j}\|\log\boldsymbol{y}_i-\log\boldsymbol{y}_j\|_{\infty}.
\end{equation}

\begin{proposition} \label{vector midpoint}
The lower bound \cref{lower vector} is attained by $\boldsymbol{x}^{\star}=(x^a)\in\mathbb{R}^n_+$ defined by
$x^a=\left(\min_i y_i^a \cdot \max_i y_i^a\right)^{1/2}$.
\end{proposition}

\begin{proof}
Note that 
\begin{equation*}
l = \frac{1}{2}\max_{i,j}\;\max_a \bigg|\log\frac{y_i^a}{y_j^a}\bigg| = \frac{1}{2}\max_a \bigg |\log\frac{\max_i y_i^a}{\min_j y_j^a}\bigg|.
\end{equation*}
With $\boldsymbol{x}^{\star}$ as defined in \cref{vector midpoint}, we have
\begin{align*}
f(\boldsymbol{x}^{\star})&=\max_{k,a} \bigg| \frac{1}{2}\log\left(\min_i y_i^a \cdot\max_i y_i^a\right)-\log y_k^a\bigg| = \max_{k,a}\bigg|\frac{1}{2}\log\left(\frac{\min_i y_i^a \cdot \max_i y_i^a}{y_k^a\cdot y_k^a}\right)\bigg|  \\
& = \max_a \bigg|\frac{1}{2}\log\left(\frac{\min_i y_i^a \cdot \max_i y_i^a}{(\min_iy_i^a)^2}\right)\bigg|  = \max_a\bigg|\frac{1}{2}\log\left(\frac{\max_i y_i^a}{\min_j y_j^a}\right)\bigg| = l.
\end{align*}
\end{proof}

\begin{remark}
The midrange problem does not generally have a unique solution in the vector case as can be readily seen through simple examples. For instance, the problem in $\mathbb{R}_+^2$ with $N=2$ and $\boldsymbol{y}_1=(a,1)$, $\boldsymbol{y}_2=(1/a,1)$ for some $a>1$ has the solution $\boldsymbol{x}=(1,s)$ for any $s$ satisfying $1/a<s<a$. 
\end{remark}

\subsection{Geometric midranges via convex optimization}
\label{subsec: quasiconvex}

The geometric matrix midrange problem \cref{N geo midrange 1} can be written as
		   \begin{equation*}
	\label{eq:opt_lambda}
	\min_{X\succ 0}  \max_{i}  \{|\log (\lambda_{\min}(Y_i^{-1/2} X Y_i^{-1/2}))|,|\log( \lambda_{\max}(Y_i^{-1/2} X Y_i^{-1/2}))| \},
	\end{equation*} 
which has the equivalent epigraph formulation
\begin{equation*}
	\begin{cases}
	 \min_{X \succ 0,\,t\in\mathbb{R}}    t\\
	 -t \leq   \log(\lambda_{\max}(Y_i^{-\frac{1}{2}} X Y_i^{-\frac{1}{2}})) \leq t \ \text{ for all }\ i\\
	 -t \leq  \log(\lambda_{\min}(Y_j^{-\frac{1}{2}} X Y_j^{-\frac{1}{2}})) \leq t \ \text{ for all }\ j.
	\end{cases}
	\end{equation*} 
This can be rewritten as the \emph{quasiconvex} problem
\begin{equation}  \label{matrix midrange quasi}
\begin{cases}
\min_{X \succ 0,\,t\in\mathbb{R}}\; t \\
e^{-t}Y_i \preceq X \preceq e^{t}Y_i.
\end{cases}
\end{equation}
While this problem is not convex due to the presence of the $\log$ function, the feasibility condition $e^{-t}Y_i \preceq X \preceq e^{t}Y_i$ is convex for \emph{fixed} $t$ and can be solved using standard convex optimization packages such as CVX \cite{cvx}. Given a $t$ that is greater than or equal to the optimum value $t^\opts=\min_{X\succ 0}  \max_i \; d_{\infty}(X,Y_i)$, we can solve \cref{matrix midrange quasi} using the bisection method \cite{Boyd2004} by successively solving the feasibility problem as we effectively decrease $t$. In the bisection method it is desirable to have a good estimate for the initial $t$ as the successive reductions in $t$ can be quite slow. In particular, if the lower bound $l=\frac{1}{2}\operatorname{diam}_{\infty}(\{Y_i\})$ is attained as in the vector case, then we can solve \cref{matrix midrange quasi} in one step by taking $t=l$ and solving the feasibility condition once. However, rather remarkably, numerical examples show that unlike the scalar and vector case, the lower bound $l$ is not always attained in the geometric matrix midrange problem. 

\begin{proposition} \label{lower bound matrix}
The lower bound $l=\frac{1}{2}\operatorname{diam}_{\infty}(\{Y_i\})$ is not necessarily attained in \cref{N geo midrange 1}.
\end{proposition}

\begin{proof}
Consider the $N=3$ geometric midrange problem in $\mathbb{P}_2$ for
\begin{equation*}
Y_1= \begin{pmatrix}
0.95 & -0.6 \\
-0.6 & 1.1
\end{pmatrix} \quad 
Y_2= \begin{pmatrix}
1.0 & 0.5 \\
0.5 & 2.1
\end{pmatrix} \quad 
Y_3= \begin{pmatrix}
2.5 & -0.2 \\
-0.2 & 1.2
\end{pmatrix}. 
\end{equation*}
The lower bound $l$ is computed to be $\frac{1}{2}\operatorname{diam}_{\infty}(\{Y_1,Y_2,Y_3\})=0.7880$. On the other hand, solving the quasiconvex optimization problem \cref{matrix midrange quasi} via the bisection method yields the midrange 
\begin{equation*}
X^\opts= \begin{pmatrix}
1.3154 & -0.5321 \\
-0.5321 & 1.6217
\end{pmatrix}  
\end{equation*}
with minimum cost $t^\opts= 0.7901 > 0.7880 = l$. Indeed, we have $t^\opts=d_{\infty}(X^\opts,Y_1)=d_{\infty}(X^\opts,Y_2)=d_{\infty}(X^\opts,Y_3)$.
\end{proof}

\begin{definition} \emph{(Active matrices)}
Let $N \geq 2$ and $(X^\opts,t^\opts)$ be a solution to \cref{matrix midrange quasi}. Then $Y_j$ is called active if $d_{\infty}(X^{\opts},Y_j)=t^{\opts}$. In particular,
at least one of the following must hold:
\begin{equation*}
-\log(\lambda_{\min}(Y_j^{-\frac{1}{2}} {X}^\opts Y_j^{-\frac{1}{2}})) =  t^\opts \quad \mathrm{or} \quad \log(\lambda_{\max}(Y_j^{-\frac{1}{2}} {X}^\opts Y_j^{-\frac{1}{2}})) = t^\opts.
\end{equation*}
\end{definition}

\cref{lower bound matrix} suggests that the $N$-point matrix midrange problem is richer than the vector case in fundamental ways. While the bisection method applied to the quasiconvex problem \cref{matrix midrange quasi} offers a solution, it can be quite slow due to the need for multiple bisection steps and the requirement to compute a reasonable upper bound estimate of the optimum cost for initialization. However, it is possible to recast \cref{matrix midrange quasi} as a convex optimization problem and thereby obtain a dramatic improvement in efficiency by introducing new variables. Specifically, by introducing $\xi=e^t$ and $\tau=e^{-t}$, and adding the extra convex constraint that $1/\xi-\tau\leq 0$, we find that \cref{matrix midrange quasi} can be reformulated as the convex optimization problem
\begin{equation} \label{Convex formulation}
\begin{cases}
\min_{X\succeq 0, \, \xi\in\mathbb{R}, \, \tau\in\mathbb{R}}\;\xi \\
\tau Y_i \preceq X \preceq \xi Y_i \\
1/\xi - \tau \leq 0
\end{cases}
\end{equation}
which can generally be solved much more efficiently than \cref{matrix midrange quasi} using standard convex optimization techniques and software packages. 

\begin{example}
As an example,  we use \cref{Convex formulation} to compute the geometric midrange of $N=1000$ real symmetric positive definite $2\times 2$ matrices. The data was generated as $Y_j=\Sigma+A_j^TA_j$, for $j=1,\cdots,N$, where the $A_j$ are $2\times 2$ matrices with normally distributed entries and $\Sigma \succ 0$ is a fixed matrix. The data matrices can be represented as points in a cone in $\mathbb{R}^3$ via the the bijection
\begin{equation*}
	\begin{pmatrix} a & b \\ b & c \end{pmatrix} \mapsto \left( \sqrt{2}b, \frac{1}{\sqrt{2}}(a-c), \frac{1}{\sqrt{2}}(a+c)\right)
\end{equation*}
\end{example}
as described in \cite{Mostajeran2018}. \cref{fig N point} shows a visualization of the results of the computation in $\mathbb{R}^3$ from two perspectives. The surrounding open cone represents the boundary of the set of real symmetric positive definite matrices and the cloud of points in gray are the data points $Y_i$. At optimum, we find that there are 4 active points that are highlighted in red, two of which nearly coincide in the figure. The geometric midrange is highlighted in blue and is the center of the Thompson sphere of radius $t^{\opts}$ that defines the smallest enclosing Thompson ball of the data. Note how the active points lie on this sphere. It is interesting that the Thompson ball is the intersection of two cones in this representation. For the sake of comparison, the Karcher mean of the data is also included as a solid black point.

\begin{figure}
\centering
\includegraphics[width=1\linewidth]{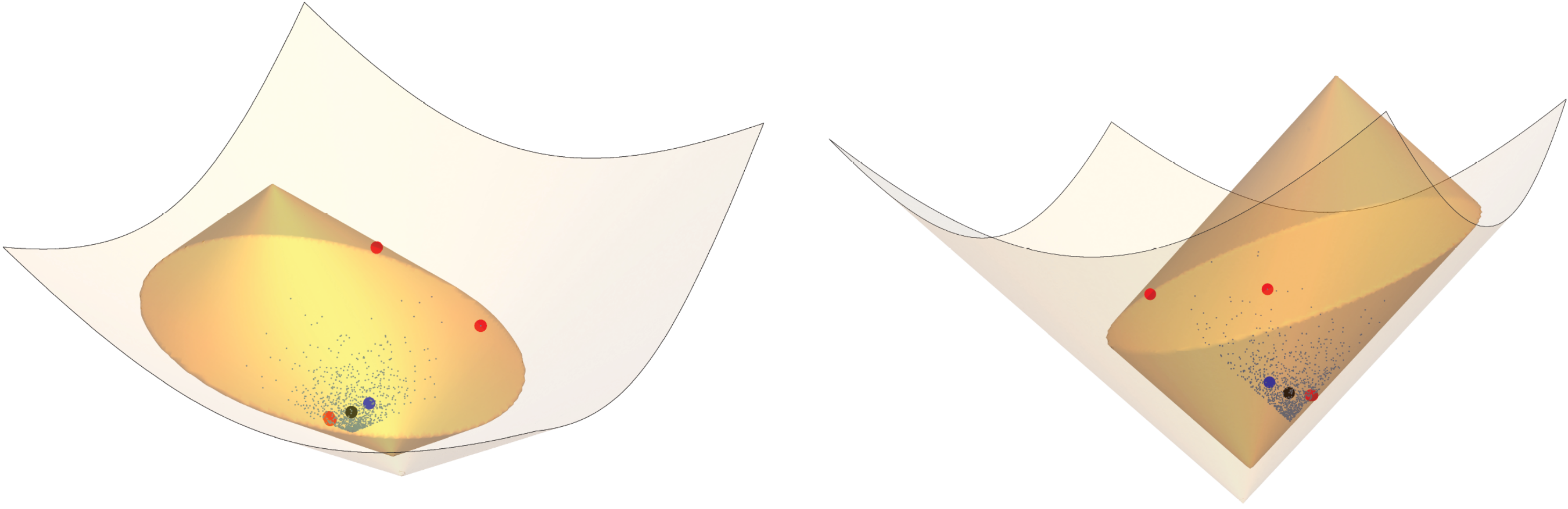}
  \caption{The geometric midrange (blue point), Karcher mean (black point), active data matrices (red points), and smallest enclosing Thompson ball for an example with $N=1000$ matrices depicted in $\mathbb{R}^3$ from two perspectives.
   }
   \label{fig N point}
\end{figure}

\begin{remark}
The preceding analysis provides an interesting example of a nonconvex optimization problem that admits a reformulation as a convex optimization problem in the Euclidean sense through a nonlinear change of coordinates.
\end{remark}

While the convex formulation \cref{Convex formulation} offers a dramatic improvement to the bisection algorithm applied to the quasiconvex formulation of the problem, we
expect that yet more efficient solutions to the problem can be found. In particular, conventional SDP-solvers are based on interior point methods with fast convergence, but high cost per iteration \cite{peaucelle2002user,toh2004implementation}, which makes them less suitable for matrices of larger size. Alternatively, one may consider so-called proximal splitting methods such as alternating projections, Douglas-Rachford, or the alternating direction method of multipliers (ADMM) \cite{BoydDistributed, combettes2011proximal, douglas1956numerical, eckstein1992douglas} applied to \cref{matrix midrange quasi}, which have cheap cost per iteration. Unfortunately, these methods tend to have poor convergence properties when the optimal solution is an intersection point of the boundaries of two convex sets with a small intersection angle \cite{faelt2017line,faelt2017optimal}. Indeed, our numerical experiments indicate that the rates of convergence of such methods degrade as $t$ gets close to the true minimum.  This can be expected as each $Y_i$ in \cref{matrix midrange quasi} defines a bounding box for $X$ through the inequality constraint and $X$ cannot be an interior point to all of them. Ideally, an efficient algorithm for solving this problem would principally rely on the computation of dominant generalized eigenpairs as in the $N=2$ case for which very efficient algorithms exist. In the next subsection, we will consider the optimality conditions for the geometric midrange problem in more detail. Before doing so, we note the following special case for which the $N$-point midrange problem reduces to the 2-point problem as in the scalar case.

\begin{proposition} \label{ordered midrange}
If $Y_1,\dots,Y_N$ are such that $Y_1\preceq Y_i \preceq Y_N$ for all $i = 1, \dots, N$, then the geometric midrange of $\{Y_i\}$  is given by the set of $d_{\infty}$-midpoints of $Y_1$ and $Y_N$. 
\end{proposition}
\begin{proof}
The ordering $Y_1 \preceq Y_i \preceq Y_N$ means that the intersection of the feasibility constraints $e^{-t}Y_i \preceq X \preceq e^t Y_i$ in the epigraph formulation \cref{matrix midrange quasi} is simply 
\begin{equation*}
e^{-t}Y_N \preceq X \preceq e^t Y_1.
\end{equation*}
Thus, the optimization problem is unchanged following the elimination of all $Y_i$ for $i\neq 1, N$. Hence, the problem is equivalent to the midrange problem for $\{Y_1,Y_N\}$ and is solved by any $d_{\infty}$-midpoint of this pair. Furthermore, the lower bound $l=\frac{1}{2}\operatorname{diam}_{\infty}(\{Y_i\})=\frac{1}{2}d_{\infty}(Y_1,Y_N)$ is trivially attained.
\end{proof}

\begin{remark}
Note that in the above we do not assume an order relation between $Y_i$ and $Y_j$ for $i,j\neq 1,N$. The value of this result lies in the insight that it provides in how and
why the matrix $N$-point midrange problem diverges from the scalar and vector
case. Fundamentally, no order relation need exist between a pair of matrices, whereas in
the scalar case such an ordering is always possible, and similarly an unambiguous
ordering is possible at the level of coordinates for vectors.
\end{remark}

\subsection{Necessary optimality conditions}\label{subsec: optimality conditions}

Finally, we prove a number of results on the optimality conditions of the geometric matrix midrange problem and the connection between the attainment of the lower bound $l=\frac{1}{2}\operatorname{diam}_{\infty}(\{Y_i\})$ and the number of active matrices at optimum.

	\begin{proposition}
		\label{prop:active}
		Let $N \geq 2$ and $(X^\star,t^\opts)$ be a solution to \cref{matrix midrange quasi}. Then there exist distinct $i^\opts, j^\opts \in \{1,
		\dots, N \}$ such that 
		\begin{equation*}
		\log(\lambda_{\max}(Y_{i^{\opts}}^{-\frac{1}{2}} X^\opts Y_{i^{\opts}}^{-\frac{1}{2}})) = -\log(\lambda_{\min}(Y_{j^{\opts}}^{-\frac{1}{2}} X^\opts Y_{j^{\opts}}^{-\frac{1}{2}})) = t^\opts
		\end{equation*}
	\end{proposition}
\begin{proof}
	By the definition of $(X^\opts,t^\opts)$, there exists at least one index $i^\opts$ or $j^\opts$ such that $|\log(\lambda_{\max}(Y_{i^{\opts}}^{-\frac{1}{2}} X^\opts Y_{i^{\opts}}^{-\frac{1}{2}}))| = t^\opts$ or $|\log(\lambda_{\min}(Y_{j^{\opts}}^{-\frac{1}{2}} X^\opts Y_{j^{\opts}}^{-\frac{1}{2}}))| = t^\opts$. In particular, for such $i^\opts$ and $j^\opts$ it must hold that 
	\begin{equation*}
	\log(\lambda_{\max}(Y_{i^{\opts}}^{-\frac{1}{2}} X^\opts Y_{i^{\opts}}^{-\frac{1}{2}})), -\log(\lambda_{\min}(Y_{j^{\opts}}^{-\frac{1}{2}} X^\opts Y_{j^{\opts}}^{-\frac{1}{2}})) \geq 0.
	\end{equation*}
	Next we will show that if there exists only one $i^\opts$ or $j^\opts$, then $(X^\opts,t^\opts)$ would not be a solution. To this end, assume that no index such as $j^\opts$ exists, so that 
	\begin{equation*}
	\begin{aligned}
	|\log(\lambda_{\min}(Y_j^{-\frac{1}{2}} X^\opts Y_j^{-\frac{1}{2}}))| < \log(\lambda_{\max}(Y_{i^{\opts}}^{-\frac{1}{2}} X^\opts Y_{i^{\opts}}^{-\frac{1}{2}})) =t^\opts \text{ for all } \ j\\
	 \log(\lambda_{\max}(Y_{i}^{-\frac{1}{2}} X^\opts Y_{i}^{-\frac{1}{2}})) \leq \log(\lambda_{\max}(Y_{i^{\opts}}^{-\frac{1}{2}} X^\opts Y_{i^{\opts}}^{-\frac{1}{2}})) = t^\opts \text{ for all } \ i.
	\end{aligned}
	\end{equation*}
	Then for sufficiently large $0< k < 1$ and $\tilde{X}^\opts := kX^\opts$, it holds that 
		\begin{equation*}
	\begin{aligned}
	|\log(\lambda_{\min}(Y_j^{-\frac{1}{2}} \tilde{X}^\opts Y_j^{-\frac{1}{2}}))| < \log(\lambda_{\max}(Y_{i^{\opts}}^{-\frac{1}{2}} \tilde{X}^\opts Y_{i^{\opts}}^{-\frac{1}{2}})) < t^\opts \text{ for all } \ j\\
		 \log(\lambda_{\max}(Y_{i}^{-\frac{1}{2}} \tilde{X}^\opts Y_{i}^{-\frac{1}{2}})) \leq \log(\lambda_{\max}(Y_{i^{\opts}}^{-\frac{1}{2}} \tilde{X}^\opts Y_{i^{\opts}}^{-\frac{1}{2}})) < t^\opts \text{ for all } \ i,
	\end{aligned}
	\end{equation*}
	which would mean that $\tilde{X}^\opts$ is a feasible solution of smaller cost than $t^\opts$. Analogously, it follows that there always exists an index $i^\opts$ with the required property. 
\end{proof}

\begin{proposition}
	Recall the $N=2$ geometric midrange problem for $Y_1$ and $Y_2$ in $\mathbb{P}_n$. Set $\alpha :=  \lambda_{\max}(Y_1^{-\frac{1}{2}} Y_2 Y_1^{-\frac{1}{2}})$ and $\beta := \lambda_{\min}(Y_1^{-\frac{1}{2}} Y_2 Y_1^{-\frac{1}{2}})$. If $\alpha \neq \beta$, then
	\begin{equation}\label{ansatz}
	X^\opts = \frac{ \sqrt{\alpha \beta} Y_1 + Y_2}{\sqrt{\alpha} +\sqrt{\beta}} = Y_1*Y_2
	\end{equation}
	is the only midrange of $\{Y_1,Y_2\}$ in $\{k_1 Y_1+k_2 Y_2: k_1,k_2 \geq 0 \}$ for which the following is satisfied:
	\begin{equation} \label{simultaneous opt}
	\begin{cases}
	\log(\lambda_{\max}(Y_1^{-\frac{1}{2}} X^\opts Y_1^{-\frac{1}{2}})) = -\log(\lambda_{\min}(Y_2^{-\frac{1}{2}} X^\opts Y_2^{-\frac{1}{2}})) \\
	\log(\lambda_{\max}(Y_2^{-\frac{1}{2}} X^\opts Y_2^{-\frac{1}{2}})) = -\log(\lambda_{\min}(Y_1^{-\frac{1}{2}} X^\opts Y_1^{-\frac{1}{2}})). 
	\end{cases}
	\end{equation}
	The optimal cost to \cref{matrix midrange quasi} is given by $t^\opts = \frac{1}{2}\max \{ |\log(\alpha)|, |\log(\beta)| \}$. Furthermore, if $Y_2 V = Y_1 VD$ is a generalized eigenvalue decomposition such that $V^* Y_1  V = I$ and $D$ is diagonal, then $V^* X^\opts V$ is diagonal. 
\end{proposition}
\begin{proof}
	Using the linear ansatz $X^\opts = k_1 Y_1+k_2 Y_2$, we obtain:
	\begin{align*}
	\lambda_{\max}(Y_1^{-\frac{1}{2}} X^\opts Y_1^{-\frac{1}{2}}) &= k_1 + k_2 \lambda_{\max}(Y_1^{-\frac{1}{2}} Y_2 Y_1^{-\frac{1}{2}}) = k_1 + k_2 \alpha \\
	\lambda_{\max}(Y_2^{-\frac{1}{2}} X^\opts Y_2^{-\frac{1}{2}}) &= k_1  \lambda_{\max}(Y_2^{-\frac{1}{2}} Y_1 Y_2^{-\frac{1}{2}})  + k_2 = \frac{k_1}{\beta} + k_2\\
	\lambda_{\min}(Y_1^{-\frac{1}{2}} X^\opts Y_1^{-\frac{1}{2}}) &= k_1 + k_2 \lambda_{\min}(Y_1^{-\frac{1}{2}} Y_2 Y_1^{-\frac{1}{2}}) = k_1 + k_2{\beta}\\
	\lambda_{\min}(Y_2^{-\frac{1}{2}} X^\opts Y_2^{-\frac{1}{2}}) &= k_1  \lambda_{\min}(Y_2^{-\frac{1}{2}} Y_1 Y_2^{-\frac{1}{2}})  + k_2 = \frac{k_1}{\alpha} + k_2.
	\end{align*}
	Substituting these expressions into \cref{simultaneous opt}, 
	we find that
	\begin{equation} \label{eq:mid_1} 
	k_1 + k_2 \alpha = \frac{1}{\frac{k_1}{\alpha} + k_2} \quad \textrm{and} \quad
	\frac{k_1}{\beta} + k_2 = \frac{1}{k_1 + k_2{\beta}}.
	\end{equation}
	Hence
	\begin{align*}
	& (k_1 + k_2 \alpha)\left(\frac{k_1}{\alpha} + k_2\right) = \left(\frac{k_1}{\beta} + k_2\right)(k_1 + k_2{\beta}),
	\end{align*}
	which is equivalent to $(k_2^2\alpha\beta-k_1^2)(\alpha-\beta) = 0$ and 
	implies that $k_1^2=k_2^2\alpha\beta$ since $\alpha\neq \beta$.
    Substituting $k_1 = \sqrt{\alpha \beta} k_2$ into the first equation of \cref{eq:mid_1} gives $k_2^2 \beta + 2 k_2^2  \sqrt{\alpha \beta} + k_2^2 \alpha = k_2^2(\sqrt{\alpha} + \sqrt{\beta})^2= 1$
	and thus \cref{ansatz}. That the cost is given by $t^\opts$ is trivial and optimality of $X^\opts$ follows by the attainment of the lower bound. Finally, $V^* X^\opts V$ is diagonal by \cref{prop:active}.
\end{proof}

\begin{remark}
Note that if $\alpha = \beta$ in the statement of the previous proposition, then $Y_1^{-1/2}Y_2Y_1^{-1/2}=\alpha I$, which is equivalent to $Y_2 = \alpha Y_1$. The midrange $Y_1\star Y_2 = \sqrt{\alpha}Y_1$ can then be obtained as a conic combination of $Y_1$ and $Y_2=\alpha Y_1$ in a non-unique way.
\end{remark}

In the remainder of this section, we explore the significance of the number of active points at optimum for the attainment of the lower bound of \cref{matrix midrange quasi} when $N\geq 2$.

\begin{lemma}
	\label{lem:active}
	Let $N \geq 2$ and $(X^\opts,t^\opts)$ be a solution to \cref{matrix midrange quasi}. Then the following are equivalent:
		\begin{enumerate}
		\item $Y_j$ is active with $\log\left(\lambda_{\max}\left(Y_j^{-\frac{1}{2}} {X}^\opts Y_j^{-\frac{1}{2}}\right)\right) = t^\opts$
		\item $\lambda_{\min}(e^{t^\opts}Y_j - X^\opts) = 0$
		\item $\nexists \varepsilon > 0: X^\opts \preceq e^{t^\opts} Y_j - \varepsilon I$
	\end{enumerate}
Analogously, we have the equivalences: 
	\begin{enumerate}
	\item $Y_j$ is active with $-\log\left(\lambda_{\min}\left(Y_j^{-\frac{1}{2}} {X}^\opts Y_j^{-\frac{1}{2}}\right)\right) = t^\opts$
	\item $\lambda_{\max}(X^\opts-e^{-t^\opts}Y_j) = 0$
	\item $\nexists\varepsilon > 0: X^\opts \succeq e^{t^\opts} Y_j + \varepsilon I$
\end{enumerate}
\end{lemma}

\begin{lemma}
	\label{lem:blkdiag}
	Let $D = \diag(d_1,\dots,d_n)$ with $d_1 \leq \dots \leq d_n$ and $D \preceq X \preceq d_n I$. Then,
	\begin{equation*}
	X = \begin{pmatrix}
	X_{11} & 0\\
	0 &	d_n
	\end{pmatrix}
	\end{equation*}
	and thus $d_n = \lambda_{\max}(X) \geq \lambda_{\max}(X_{11})$.
\end{lemma}
\begin{proof}
	From the inequality it follows that $X_{nn} = d_n$ and $\lambda_{\max}(X) \leq d_n$.  Thus, by Courant-Fischer, $e_n=(0,0,\dots,1)$ is an eigenvector of $X$ with eigenvalue $d_n$ and thus $X$ has the required form. 
\end{proof}

\begin{proposition}
	Let $(X^\opts,t^\opts)$ be a solution to \cref{matrix midrange quasi} and 
	\begin{equation*}
	t^\opts =  \frac{1}{2}\log\left(\lambda_{\max}\left(Y_1^{-\frac{1}{2}} Y_2 Y_1^{-\frac{1}{2}}\right)\right).
	\end{equation*}
	Then, $Y_1$ and $Y_2$ are active with
	\begin{equation*}
	\log\left(\lambda_{\max}\left(Y_1^{-\frac{1}{2}} {X}^\opts Y_1^{-\frac{1}{2}}\right)\right) = t^\opts = -\log\left(\lambda_{\min}\left(Y_2^{-\frac{1}{2}} {X}^\opts Y_2^{-\frac{1}{2}}\right)\right).
	\end{equation*}
	Further, if $Y_2 V = Y_1 VD$ is a generalized eigenvalue decomposition such that $V^* Y_1  V = I$ and $D = \diag(d_1,\dots,d_n)$ with $d_1 \leq \dots \leq d_n = \lambda_{\max}\left(Y_1^{-\frac{1}{2}} Y_2 Y_1^{-\frac{1}{2}}\right)$, then 
	\begin{equation*}
	V^*X^\opts V =  \begin{pmatrix}
	X_{11} & 0\\
	0 &	\sqrt{d_n}
	\end{pmatrix}.
	\end{equation*}
\end{proposition}
\begin{proof}
	Let $X_V := V^* X^\opts V$. We first show that $X_V$ has the claimed structure. To this end, note that by \cref{matrix midrange quasi} 
	\begin{equation*}
	e^{-t^\opts} I \preceq X_V \preceq e^{t^\opts} I \quad \textrm{and} \quad
	e^{-t^\opts} D \preceq X_V \preceq e^{t^\opts} D,
	\end{equation*}
which implies that $e^{-t^\opts} D \preceq X_V \preceq e^{t^\opts} I$
	with $e^{t^\opts} =  \sqrt{d_n}$. Therefore, \cref{lem:blkdiag} implies the required structure for $X_V$. Then by \cref{lem:active} it follows that $e^{t^\opts} I$ and $e^{-t^\opts} D$ are active matrices for $(X_V,t^\opts)$ and thus $Y_1$ and $Y_2$ are active matrices for $(X^\opts,t^\opts)$. Then by \cref{lem:blkdiag} we can conclude the remaining claim.
\end{proof}

\begin{proposition} \label{attainment}
If there are only two active matrices at an optimum $(X^\opts,t^\opts)$ of \cref{matrix midrange quasi}, then the lower bound $l=\frac{1}{2}\operatorname{diam}_{\infty}(\{Y_i\})$ is attained.
\end{proposition}

\begin{proof}
Suppose that $Y_1$ and $Y_2$ are the only two active matrices at $(X^{\opts},t^\opts)$ and assume that the lower bound $l$ is not attained so that $l < t^\opts$. Denote the geodesic \cref{R geodesic} from $X^\opts$ to the $d_{\infty}$-midpoint $Y_1*Y_2$ of $Y_1$ and $Y_2$ by $\gamma_{\mathcal{G}}(s)=\gamma_{\mathcal{G}}(s,X^\opts,Y_1*Y_2)$.
By the geodesic convexity of $(\mathbb{P}_n,d_{\infty})$, the function $s\mapsto d_{\infty}(Y_j,\gamma_{\mathcal{G}}(s,X^\opts,Y_1*Y_2))$ is convex for $j=1,2$. Thus, we have
\begin{align*}
d_{\infty}(Y_j,\gamma_{\mathcal{G}}(s,X^\opts,Y_1*Y_2)) &\leq (1-s)d_{\infty}(Y_j,X^\opts)+sd_{\infty}(Y_j,Y_1*Y_2) \\
& = (1-s)t^\opts + \frac{s}{2}d_{\infty}(Y_1,Y_2) \\
&\leq (1-s)t^\opts + sl  \\
&<(1-s)t^\opts + s t^\opts = t^\opts,
\end{align*}
for any $s>0$ and $j=1,2$. As all matrices other than $Y_1$ and $Y_2$ are inactive at $s=0$, we can achieve a local reduction in the cost function by moving a sufficiently small $s>0$ along the geodesic from $X^\opts$ to $Y_1*Y_2$, which would contradict the optimality of $(X^\opts,t^\opts)$. Thus, we  have $t^\opts=l$.
\end{proof}

\begin{remark} 
Note that although \cref{attainment} implies that the optimum $X^\opts$ will lie in the $d_{\infty}$-midpoint set of the active pair of matrices, it does not imply that any $d_{\infty}$-midpoint of $Y_1,Y_2$ will be a solution. In particular, $Y_1*Y_2$ may not be a solution even if the only active matrices at optimum are $Y_1$ and $Y_2$ since $Y_1*Y_2$ may fail to satisfy one or more of the constraints in \cref{matrix midrange quasi}. This is in contrast to the scenario in \cref{ordered midrange}, where any midrange of $Y_1$ and $Y_N$ will be a solution.
\end{remark}

\section{Conclusion}
\label{sec:conclusion}

We have introduced a theory of geometric midrange statistics for positive definite Hermitian matrices within an optimization framework. We have also established a number of key results including bounds on the optimization problem as well as necessary conditions for optimality. Furthermore, a solution to the $N$-point problem is offered via convex optimization. Special consideration has been given to the $2$-point midrange problem, which was studied in detail from a number of complementary perspectives. The existence of solutions to the $2$-point problem that can be computed using only extremal generalized eigenvalues has significant implications for computational scalability of matrix midrange statistics. We expect this work to offer a solid foundation for future research in statistics based on Thompson geometry and related topics such as $K$-midranges \cite{k-midranges,clustering2006} for matrix-valued data. The development of a fast algorithm for the computation of a midrange of $N$ matrices would be an important step in this direction, with weighted inductive schemes and stochastic algorithms offering a promising angle of attack.

\section*{Acknowledgments}
We are most grateful to Yurii Nesterov for suggesting the change of variables that facilitated the conversion of the quasiconvex formulation of the $N$-point geometric matrix midrange problem to a convex optimization problem.

\bibliographystyle{siamplain}
\bibliography{references}
\end{document}